\documentclass[11pt, oneside]{article}   	
\usepackage[utf8]{inputenc}
\usepackage{geometry}                		
\usepackage{graphicx}				
\usepackage{amssymb, amsmath, amscd, amsthm}

\usepackage[english]{babel}
\usepackage[T1]{fontenc}
\usepackage{hyperref}
\usepackage{cancel}

\newtheorem{lemme}{Lemma}[section]
\newtheorem{corollaire}{Corollary}[section]
\newtheorem{proposition}{Proposition}[section]
\newtheorem{definition}{Definition}[section]
\newtheorem{remarque}{Remark}[section]
\newtheorem{theorem}{Theorem}[section]

\def\A{\mathcal{A}}
\def\B{\mathcal{B}}
\def\E{\mathcal{E}}

\title{A differential graded Lie algebra approach to non abelian extensions of associative algebras}
\author{Jean-Baptiste Gouray}
\date{}							

\begin{document}

\maketitle

\begin{abstract}
	In this paper we show that non abelian extensions of an associative algebra $\B$ by an associative algebra $\A$ can be viewed as Maurer-Cartan elements of a suitable differential graded Lie algebra $L$. In particular we show that $\mathcal{MC}(L)$, the Deligne groupoid of $L$, is in 1-1 correspondence with the non-abelian cohomology $H^2_{nab}(\B,\A)$.
\end{abstract}

\tableofcontents
	\addcontentsline{toc}{section}{Introduction}

\section*{Introduction}

Many classification problems can be reduced to some cohomology computations. One can cite derivations in terms of $H^1(A,A)$ (\cite{MR1438546} and \cite{MR1269324}), abelian extensions in terms of $H^2(A,M)$ (\cite{MR1438546} and \cite{MR1269324}), deformations in terms of $H^2(A,M)$ (\cite{MR2381782} and \cite{MR0024908} for associative and Lie), crossed modules in terms of $H^3(A,B)$ (\cite{2009arXiv0911.2861T} and \cite{MR2229486} for groups and Lie algebras), and so on. 

Among them, non-abelian extensions play a special rôle in the sense that the algebraic structure that governs them is not  exactly  a cohomology theory as it does not come from a complex. The first occurrence of such a study appeared in the setting of Lie algebras in \cite{MR3123761}. It is only recently  that the associative case has been studied in \cite{MR3545266}, building on \cite{MR0022842}.
Yaël Frégier has conjectured that similar theories should exist for most algebraic structures and suggested an approach to unify such a treatment, based on the use of differential graded Lie algebras (dgLa's).\\


 In this paper we propose to investigate the associative case using this dgLa approach suggested by Frégier. The main theorem \ref{thmprinc} states that
\[
 	H^2_{nab}(\B,\A) \simeq \mathcal{MC}(L).
\]
Here $\B$ is a given associative algebra, $\A$ is an associative algebra with which we want to extend $\B$ and $L$ is a dgLa. It is actually a sub dgLa of $(C^{\bullet+1}(\A \oplus \B,\A \oplus \B),[\,,\,],\delta)$ where $(C(\A \oplus \B,\A \oplus \B),\delta)$ is the Hochschild cohomology and $[\, , \,]$ is the Gerstenhaber bracket.\\

We  begin by recalling the results of \cite{MR3545266}, i.e. the classification of non-abelian extensions in terms of  non-abelian cohomology. We present the notions useful for the theorem $\ref{thmprinc}$ namely Deligne groupoids, Hochschild cohomology and Gerstenhaber bracket. In the second part, we show how non-abelian cohomology appears naturally in a dgLa context and we prove the main theorem \ref{thmprinc}. Finally, in the last part, we make the link with abelian extensions, we will see that it is a particular case of non-abelian extension.\\

\section{Background and useful notions}	
\subsection{Non abelian extensions and cohomology}

In this subsection we recall the notion defined in \cite{MR3545266}.

\begin{definition}
	Let $\A$ and $\B$ two associative algebras. An n\textbf{on abelian extension} $\E$ of $\B$ by $\A$ is a short exact sequence of the form (in the category of associative algebras)
	$$\begin{CD}
		0 @>>> \A @>>> \E @>>> \B @>>> 0.
	\end{CD}$$
\end{definition}

\vspace{1\baselineskip}
As explained in \cite{MR3545266}, in the framework of \textit{GE-problem} one is interested in extensions modulo the following equivalence relation. 

\begin{definition}
	Let $\E$ and $\E'$ two extensions of $\B$ by $\A$. They are called \textbf{equivalent} if there exists $\theta : \E \rightarrow \E'$ such that the following diagram commutes
	$$\begin{CD}
	0 @>>> \A @>>> \E @>>> \B @>>> 0 \\
	   @.	         @|	     @V\theta VV @|\\
	0 @>>> \A @>>> \E' @>>> \B @>>> 0.
	\end{CD}$$
\end{definition}

Such extensions are classified by non abelian cohomology. 

\begin{definition}\label{defcoho}
	 A \textbf{non abelian 2-cocycle} on $\B$ with values in $\A$ is a triplet $(\varphi, \psi, \chi )$ of linear maps $\chi : \B \otimes \B \rightarrow \A$ and $\varphi, \psi : \B \rightarrow End(\A)$ satisfying the following properties :
	 
	 \begin{eqnarray}
		\varphi_{b_1}(\varphi_{b_2}(a)) & = &\varphi_{b_1\cdot b_2}(a) + \chi(b_1,b_2) \cdot a, \label{defg}\\
		 \psi_{b_1}(\psi_{b_2}(a)) & = &\psi_{b_2\cdot b_1}(a) + a \cdot \chi(b_1,b_2),\label{defd}\\
		 \varphi_{b_1}(\psi_{b_2}(a)) & =&\psi_{b_2}(\varphi_{b_1}(a)),\label{defbi}
	\end{eqnarray}
	
	\begin{equation}
		\psi - \varphi : \B \rightarrow Der(\A), \label{defder}
	\end{equation}
and
	\begin{equation}
		 - \varphi_{b_1}(\chi(b_2,b_3))+ \chi(b_1\cdot b_2,b_3)-\chi(b_1,b_2\cdot b_3)+\psi_{b_3}(\chi(b_1,b_2)) =0. \label{defcocy}
	\end{equation}
	
One denotes by $Z^2_{nab}(\B,\A)$ the set of theses cocycles.\\
Moreover, $(\varphi, \psi, \chi)$ and $(\varphi', \psi', \chi')$ are said to be \textbf{equivalent} if there exists $\beta : B \rightarrow A$ satisfying:

	\begin{eqnarray}
		\varphi_{b}(a) &=& \varphi_b'(a) - \beta(b)\cdot a \label{def1}\\
		\psi_{b}(a) &=& \psi_{b}'(a) - a\cdot \beta(b), \label{def2}
	\end{eqnarray}
and
	\begin{equation}
		\chi'(b_1,b_2) = \chi(b_1,b_2) - \varphi_{b_1}(\beta(b_2)) - \psi_{b_2}(\beta(b)) + \beta(b_1 \cdot b_2) + \beta(b_1) \cdot \beta(b_2). \label{def3}
	\end{equation}
	
Non abelian cohomology $H^2_{nab}(\B,\A)$  is the quotient of $Z^2_{nab}(\B,\A)$ by this equivalence relation.
\end{definition}

\begin{remarque} \label{remder}
	The three conditions given in \cite{MR3545266}
	\[
	\begin{aligned}
		\psi_{b}(a_1\cdot a_2) &= a_1\cdot \psi_{b}(a_2),\\
		\varphi_{b}(a_1\cdot a_2) &= \varphi_{b}(a_1)\cdot a_2, \\
		\psi_{b}(a_1)\cdot a_2 &= a_1\cdot \varphi_{b}(a_2).
	\end{aligned}
	\] were replace by the condition $\ref{defder}$.
	Theses conditions can be interpreted as the compatibility of $\varphi$ and $\psi$ with the multiplication in $\A$.
\end{remarque}

\begin{proposition} \label{prop1}
	There is a 1-1 correspondence between classes of extensions of $\B$ by $\A$ and elements of $H^2_{nab}(\B,\A)$. In other words, $H^2_{nab}(\B,\A)$  classifies extensions of $\B$ by $\A$.
\end{proposition}

\begin{proof}
This is already shown in \cite{MR3545266}, but we propose here to give the details of the proof in order to draw a picture of the theorical background.\\
First, to a given extension $\E$ one associates a class in $H^2_{nab}(\B,\A)$ through the choice of a section. One recall that a section of $\E$ is a map $s : \B \rightarrow \E$ in 
	$$\begin{CD}
		0 @>>> \A @>>> \E @>p>> \B @>>> 0 
	\end{CD}$$
such that $ p \circ s = id_{\B}$.\\

A representative cocycle $(\varphi^s, \psi^s, \chi^s)$ is defined by :
	$$\varphi^s_b(a) := s(b)\cdot a,$$
	$$\psi^s_b(a) := a \cdot s(b),$$
and, $$ \chi^s(b_1,b_2) := s(b_1)\cdot s(b_2) - s(b_1 \cdot b_2).$$
First, one can show that $(\varphi^s, \psi^s, \chi^s)$ is, indeed, a cocycle.
\[
\begin{aligned}
	\varphi^s_{b_1}(\varphi^s_{b_2}(a)) &= s(b_1)\cdot (s(b_2) \cdot a)\\
								 &= (s(b_1) \cdot s(b_2))\cdot a - s(b_1 \cdot b_2)\cdot a + s(b_1 \cdot b_2)\cdot a\\
								 &= \chi^s(b_1,b_2) \cdot a + \varphi^s_{b_1\cdot b_2}(a),
\end{aligned}
\]
and similarly with $\psi^s$. Since $(s(b_1) \cdot a)\cdot s(b_2) = s(b_1)\cdot (a\cdot s(b_2))$ condition ($\ref{defbi}$) is satisfied too. Also, using associativity in $\E$ one can show that $\chi^s$ verifies condition ($\ref{defcocy}$).\\

It turns out that the equivalence relation on $Z^2_{nab}(\B,\A)$ makes the class independent of the choice of the section. Indeed, if one has two sections $s$ and $s'$, one defines $\beta : \B \rightarrow \A $ by $\beta := s - s' $. Then one checks that $(\varphi^{s'}, \psi^{s'}, \chi^{s'}) \overset{ \beta}{\sim}  (\varphi^s, \psi^s, \chi^s)$ :
$$
\begin{aligned}
    \varphi^{s'}_b(a) & = s'(b) \cdot a\\
         		   & = s(b) \cdot a - (s-s')(b) \cdot a \\
		   	   & = \varphi^s_b(a) - \beta(b) \cdot a.
 \end{aligned}
$$
Similar computations lead to equations ($\ref{def2}$) and ($\ref{def3}$).\\
Moreover, two equivalent extensions $\E$, $\E'$ rise to equivalent cocycles in $H^2_{nab}(\B,\A)$. Indeed, one has 
$$\begin{CD}
	0 @>>> \A @>>> \E @>>> \B @>>> 0 \\
	   @.	         @|	     @V\theta VV @|\\
	0 @>>> \A @>>> \E' @>>> \B @>>> 0 
\end{CD}$$
with sections $s : \B \rightarrow \E$ and $s' : \B \rightarrow \E'$. Consider $\beta := \theta^{-1}\circ s' -s $, then one has  $(\varphi^{s'}, \psi^{s'}, \chi^{s'}) \overset{ \beta}{\sim}  (\varphi^s, \psi^s, \chi^s)$. Indeed since the previous diagram commutes, one has that $\theta^{-1}(a) = a$ and $\theta^{-1}(e_1 \cdot_{\E'} e_2) = \theta^{-1}(e_1) \cdot_{\E} \theta^{-1}(e_2)$, hence we have
\[
\begin{aligned}
	\varphi^{s'}_b(a) &=  s'(b) \cdot_{\E'} a\\
				  &= \theta^{-1}(s'(b))\cdot_{\E} a - s(b)\cdot_{\E}a + s(b)\cdot_{\E}a\\
				  &= \beta(b) \cdot_{\E} a + \varphi^{s}_b(a),
\end{aligned}
\]
and similarly for $\psi^s$ and $\chi^s$.\\

Conversely, to a given cocycle $(\varphi, \psi, \chi)$, then one can associate it to an extension of the form 
$$
	\begin{CD}
		0 @>>> \A @>>> \A \oplus \B_{(\varphi, \psi, \chi)} @>>> \B @>>> 0 
	\end{CD}
$$
with multiplication in $\A \oplus \B_{(\varphi, \psi, \chi)}$ defined by 
$$
	m_{\E}(a_1+b_1,a_2+b_2) := m_{\A}(a_1 , a_2 )+ \varphi_{b_1}(a_2) + \psi_{b_2}(a_1) + \chi(b_1,b_2) + m_{\B}(b_1, b_2)
$$
This association is well defined in cohomology since equivalent cocycles give equivalent extensions : 
$$\begin{CD}
	0 @>>> \A @>>> \A \oplus \B_{(\varphi, \psi, \chi)} @>>> \B @>>> 0 \\
	   @.	         @|	     @V\theta VV @|\\
	0 @>>> \A @>>> \A \oplus \B_{(\varphi', \psi', \chi')}@>>> \B @>>> 0.
\end{CD}$$
\end{proof}

\begin{remarque}
	The previous defined multiplication is associative. 
	$$
	\begin{aligned}
		((a_1+b_1)\cdot (a_2+b_2 ))\cdot(a_3+b_3) &= (a_1 a_2 + \varphi_{b_1}(a_2) + \psi_{b_2}(a_1) + \chi(b_1,b_2) + b_1 b_2)\cdot(a_3+b_3)\\	
										   & = (a_1 \cdot a_2 ) \cdot a_3   + \varphi_{b_1}(a_2) \cdot a_3 + \psi_{b_2}(a_1)\cdot a_3  + \chi(b_1,b_2)\cdot a_3 + \varphi_{b_1\cdot b_2}(a_3)\\
											    &+  \psi_{b_3}(a_1 \cdot a_2 )+ \psi_{b_3}(\varphi_{b_1}(a_2)) + \psi_{b_3}(\psi_{b_2}(a_1)) + \psi_{b_3}(\chi(b_1,b_2))\\
											    &+ \chi(b_1 \cdot b_2, b_3) + (b_1 \cdot b_2)\cdot b_3.	
	\end{aligned}
	$$
	Similarly, 
	$$
	\begin{aligned}
	-(a_1+b_1)\cdot( (a_2+b_2 )\cdot(a_3+b_3)) & = - a_1 \cdot (a_2 \cdot a_3 )  - a_1\cdot \varphi_{b_2}(a_3) - a_1\cdot \psi_{b_3}(a_2)  - a_1\cdot \chi(b_2,b_3) - \psi_{b_2\cdot b_3}(a_1)\\
											    &-  \varphi_{b_1}(a_2 \cdot a_3 )- \varphi_{b_1}(\varphi_{b_2}(a_3)) - \varphi_{b_1}(\psi_{b_3}(a_2)) - \varphi_{b_1}(\chi(b_2,b_3))\\
											    &- \chi(b_1,b_2 \cdot b_3) - b_1 \cdot (b_2 \cdot b_3). 
	\end{aligned}
	$$
Thanks to equations (\ref{defg}),(\ref{defd}),(\ref{defbi}) and (\ref{defcocy}), one has 
\begin{eqnarray*}
 &&((a_1+b_1)\cdot(a_2+b_2) )\cdot(a_3+b_3) - (a_1+b_1)\cdot( (a_2+b_2) \cdot(a_3+b_3)) \\
  &=& \varphi_{b_1}(a_2)\cdot a_3 + \psi_{b_2}(a_1) \cdot a_3  +  \psi_{b_3}(a_1 \cdot a_2 ) -  \varphi_{b_1}(a_2 \cdot a_3 ) - a_1 \cdot \varphi_{b_2}(a_3) - a_1\cdot \psi_{b_3}(a_2).
\end{eqnarray*}
Thanks to remark \ref{remder}, condition (\ref{defder}) gives that this multiplication is associative.
\end{remarque}

\subsection{Maurer-Cartan elements and Deligne groupoid}

In this part we recall notion of Maurer-Cartan elements of differential graded Lie algebra (more details can be found in \cite{MR0195995} for example).

\begin{definition}
	A \textbf{differential graded Lie algebra} (dgLa) is a graded Lie algebra $(L,[ \,, \,])$ equipped a derivation $d$ such that
	\begin{enumerate}
		\item{$|da| = |a| +1 $(d is degree 1)}
		\item{$d[a,b] = [da,b] + (-1)^{|a|}[a,db]$ (d is a derivation)} 
		\item{$d^2=0$ (d is homological).}
	\end{enumerate}
\end{definition}

Then we can define the set of Maurer-Cartan elements.

\begin{definition}
	The set of \textbf{Maurer-Cartan} of the dgLa L is
	$$
		MC(L) = \{ c \in L^1 \; \; | \; \; dc + \dfrac{1}{2}[c,c] = 0 \}.
	$$
\end{definition}

We define an equivalence relation on $MC(L)$ which is called gauge equivalence relation.
\begin{definition}
	Let $c$ and $c'$  be two elements in $MC(L)$, they are \textbf{equivalent} modulo gauge equivalence relation if there exists $\beta \in L^0 $ ad-nilpotent (i.e. $\forall x \in  L, \; \exists n \in \mathbb{N} \; (ad_{\beta})^n(x) = [\beta ,[\beta ,[...,[\beta,x]...] = 0$) such that :
	$$
		c' = exp(ad_{\beta})c	+ g_{\beta}
	$$
	where 
	$$
		exp(ad_{\beta}):= \sum_{n = 0}^{\infty}\dfrac{1}{n!}(ad_{\beta})^n,
	$$
	and,
	$$
		g_{\beta} := - \sum_{n = 0}^{\infty}\dfrac{1}{(n+1)!}(ad_{\beta})^n d\beta.
	$$
\end{definition}

The \textbf{Deligne groupoid} of $L$ is defined as  $MC(L)$ modulo gauge equivalence relation, and we denote it by $\mathcal{MC}(L)$.

\subsection{Hochschild cohomology and Gerstenhaber bracket}

The dgLa used in theorem \ref{thmprinc} needs the following notions.

\begin{definition}[Hochschild cohomology]
	Let $\A$ an associative algebra and $M$ an $\A$-bimodule, then the \textbf{Hochschild cochain} $C(\A, M)$ is the space of the multilinear maps from $\A$ to $M$:
	$$
		C^n(\A,M) := Lin(\A^{\otimes n  },M),
	$$
	with the Hochschild differential, 
	\begin{eqnarray*}
		\delta f (a_1,...,a_{n+1}) &:=& a_1\cdot f(a_2,...,a_{n+1}) + \sum_{i=1}^{n} (-1)^i f(a_1,...,a_ia_{i+1},...,a_{n+1}) \\
							&&+ (-1)^{n+1}f(a_1,...,a_n)\cdot a_{n+1}.
	\end{eqnarray*} 
\end{definition}

\begin{remarque}
With this definition, equation ($\ref{defcocy}$) rewrite $-\delta \chi = 0$, but only if $\varphi$ and $\psi$ define a structure of $\B$-bimodule on $\A$. Note that equations (\ref{defg}),(\ref{defd}), and (\ref{defbi})) give almost a structure of $\B$-bimodule on $\A$. Also, equation ($\ref{def3}$) expresses that $\chi'$ and $\chi$ differ by a Hochschild coboundary.
\end{remarque}

\begin{definition}[Gerstenhaber bracket]
	Let $f \in Lin(\A^{\otimes m +1},\A)$ and $g \in Lin(\A^{\otimes n +1},\A)$, ones defines 
	\begin{eqnarray*}
		f \circ_i g := f( id_{\A}^{(i-1)} \otimes g \otimes id_{\A}^{(m-i+1)}),
	\end{eqnarray*}
	and,
	\begin{eqnarray*}
		f \circ g := \sum_{i=1}^{m+1} (-1)^{n(i+1)}f \circ_i g.
	\end{eqnarray*}
	Then the \textbf{Gerstenhaber bracket} of $f$ and $g$ is 
	\begin{eqnarray*}
		[f,g] := f \circ g - (-1)^{mn}g \circ f.
	\end{eqnarray*}
\end{definition}

One can find in \cite{MR0161898} have proved the following proposition 

\begin{proposition}\label{propdoub}
	Let $\A$ an associative algebra, then $(C^{\bullet + 1}(\A,\A),[\;,\;], \delta)$ is a differential graded Lie algebra (dgLa). 
\end{proposition}
It has been show in \cite{MR0161898} if one sees the multiplication of $\A$ as an element in $C^2(\A,\A)$ then we have $\delta f = (-1)^{n-1} [m_{\A}, f]$ (where $m_{\A}$ is the multiplication in $\A$ and $f \in C^n(\A,\A)$).

\section{Non-abelian extensions in terms of Deligne groupoid }

The aim of this section is to prove theorem $\ref{thmprinc}$. We proceed in two times. The first time is to prove corollary $\ref{corollairecocycle}$ which states that the set of non-abelian cocycles are in bijection with the Maurer-Cartan elements of a dgLa $L$ (that we introduce in subsection 2.1). In the second time, we show that equivalence relation on $Z^2(\B,\A)$ can be interpreted as gauge equivalence relation on $MC(L)$. Theses two steps give us the theorem $\ref{thmprinc}$.

\subsection{Non abelian cocycles as Maurer-Cartan elements}

Firstly, one characterizes the (associative) multiplication $m$ of an extension $\E$ of $\B$ by $\A$. One has that 
$$\begin{CD}
	0 @>>> \A @>>> \E @>p>> \B @>>> 0 \\
	   @.	         @|	     @|        @VV\sim V\\
	0 @>>> A @>>> A\oplus B @>P>> B @>>> 0, 
\end{CD}$$
with a section $s$ (i.e. $p\circ s = id_{\B}$), and where $A$ is the vector space image of $\A$ in $\E$ and $ B = s(\B)$ an arbitrary supplementary of $A$ in $\E$.\\
We define a handy notation to compute the components of $m$. Consider the canonical projection 
$$
	P_{X_1...X_n} : (A \oplus B)^{\otimes^n} \rightarrow X_1 \otimes ... \otimes X_n,
$$
where $X_i \in \{ A , B \}$.
And let $L$ be a linear map $L : (A \oplus B)^{\otimes^n} \rightarrow A \oplus B$ one denotes  
$$
	L_{X_1...X_n}^{X_{n+1}} := P_{X_{n+1}}\circ L \circ i \circ P_{X_1...X_n}
$$
where $i : X_1 \otimes ... \otimes X_n \rightarrow (A \oplus B)^{\otimes^n} $  is the inclusion.\\

Now, one can compute the components of $m$. 
\begin{lemme}
	\begin{itemize}
		\item{$m_{AB}^{B} = m_{BA}^B = m_{AA}^B = 0 $.}
		\item{One identifies $m_{BB}^B$ with multiplication in $\B$, and $m_{AA}^{A}$ with multiplication in $\A$.}
		\item{Finally, to make the connection with definition $\ref{defcoho}$ one can introduce the notation $m_{BB}^{A} = \chi $, $m_{BA}^{A} = \varphi $ and $m_{AB}^{A} = \psi $.}
	\end{itemize}
\end{lemme}

\begin{proof}
One has 
\[
	m_{AB}^{B}(a_1 + b_1, a_2 + b_2 ) := P(m(a_1,b_2)).					
\]
But $P$ is a morphism of algebras with kernel $A$, on obtain that
\[
\begin{aligned}
	m_{AB}^{B}(a_1 + b_1, a_2 + b_2 ) &=P(a_1)\cdot P(b_2))\\
								&= 0 \cdot P(b_2) \\
								&= 0.
\end{aligned}						
\]
Analogously, one has $m_{BA}^B = m_{AA}^B = 0 $.\\
Since the projection $pr_{|B}$ has inverse $s$, we can identified $m_{BB}^B$ with multiplication in $\B$ by conjugaison. And similarly for $m_{AA}^{A}$ and multiplication in $\A$.
\end{proof}

The multiplication $m$ is associative so its associator vanishes. The vanishing of this associator is equivalent to the vanishing of all its components.
\begin{proposition} \label{propas}
	The vanishing of the components of the associator of $m$ gives us :
	\begin{description}
	\item{$ As_{BBB}^{B}$ : $m_{\B}$ is associative,}
	\item{$ As_{BBA}^{A}$ : $A$ is a left twisted $\B$-module with action $\varphi$ (i.e. $\varphi$ satisfies equation $(\ref{defg})$),}
	\item{$ As_{ABB}^{A}$ : $A$ is a right twisted $\B$-module with action $\psi$ (i.e. $\psi$ satisfies equation $(\ref{defd})$),}
	\item{$ As_{BAB}^{A}$ : $\varphi$, $\psi$ verify equation (\ref{defbi}) (i.e. $A$ is a twisted $\B$-bimodule)}
	\item{$ As_{BBB}^{A}$ : $\chi$ is a "Hochschild" cocycle,}
	\item{$ As_{AAB}^{A} + As_{ABA}^{A} + As_{BAA}^{A}$ : $\psi - \varphi : \B \rightarrow Der(\A)$}
	\item{$ As_{AAA}^{A}$ : $m_{\A}$ is associative.}
	\end{description}
\end{proposition}

\begin{proof}
Let be $e = (e_1, e_2, e_3) \in \E^3$ and $e_i = a_i + b_i$. On has

\begin{eqnarray*}
	As_{BBB}^{B}(e) &= &  m_{BB}^B(m_{BB}^B(e_1,e_2),e_3) +\cancel{m_{AB}^B}(m_{BB}^A(e_1,e_2),e_3) \\
					&&- m_{BB}^B(e_1,m_{BB}^B(e_2,e_3)) - \cancel{m_{BA}^B}(e_1,m_{BB}^A(e_2,e_3))\\
			      & = & (b_1\cdot b_2)\cdot b_3 - b_1\cdot(b_2\cdot b_3) .
\end{eqnarray*}	
\begin{eqnarray*}	      
	As_{BBA}^{A} (e) & = & \underbrace{m_{BA}^A(m_{BB}^B(e_1,e_2),e_3)}_{\Box_1} \underbrace{ + m_{AA}^A(m_{BB}^A(e_1,e_2),e_3)}_{\Box_2}\\
					&& - m_{BB}^A(e_1, \cancel{m_{BA}^B}(e_2,e_3))  \underbrace{- m_{BA}^A(e_1,m_{BA}^A(e_2,e_3))}_{\Box_3}\\
			       & = &	  \varphi_{b_1\cdot b_2}(a_3) + \chi(b_1,b_2)\cdot a_3 - \varphi_{b_1}(\varphi_{b_2}(a_3)).
\end{eqnarray*}
\begin{eqnarray*}
	As_{ABB}^{A}(e) &=& m_{BB}^A(\cancel{m_{AB}^B}(e_1,e_2),e_3) \underbrace{ + m_{AB}^A(m_{AB}^A(e_1,e_2),e_3)}_{\boxminus_3}\\
					&& \underbrace{- m_{AB}^A(e_1,m_{BB}^B(e_2,e_3))}_{\boxminus_1} \underbrace{- m_{AA}^A(e_1,m_{BB}^A(e_2,e_3))}_{\boxminus_2}\\
				&=& \psi_{b_3}(\psi_{b_2}(a_1)) - \psi_{b_2\cdot b_3}(a_1) - a_1 \cdot \chi(b_2,b_3).
\end{eqnarray*}
\begin{eqnarray*}
	As_{BAB}^{A} (e) &=& m_{BB}^A (\cancel{m_{BA}^B}(e_1,e_2),e_3) \underbrace{+ m_{AB}^A(m_{BA}^A(e_1,e_2),e_3)}_{\boxplus_2}\\
					&&- m_{BB}^A(e_1,\cancel{m_{AB}^B}(e_2,e_3))  \underbrace{- m_{BA}^A(e_1,m_{AB}^A(e_2,e_3))}_{\boxplus_1}\\
				&=& \psi_{b_3}(\varphi_{b_1}(a_2)) - \varphi_{b_1}(\psi_{b_3}(a_2)).
\end{eqnarray*}
\begin{eqnarray*}
	As_{BBB}^{A} (e) &=& \underbrace{m_{BB}^A (m_{BB}^B(e_1,e_2),e_3)}_{\bigtriangledown_1'} \underbrace{ + m_{AB}^A (m_{BB}^A (e_1,e_2),e_3)}_{\bigtriangledown_2'}\\
					&& \underbrace{- m_{BB}^A (e_1,m_{BB}^B(e_2,e_3) )}_{\bigtriangledown_1''}\underbrace{- m_{BA}^A(e_1,m_{BB}^A (e_2,e_3))}_{\bigtriangledown_2''}\\
				&=& \chi(b_1\cdot b_2,b_3) + \psi_{b_3}(\chi(b_1,b_2)) - \chi(b_1,b_2 \cdot b_3) - \varphi_{b_1}(\chi(b_2,b_3))\\
				&=& - \delta \chi(b_1,b_2,b_3). 
\end{eqnarray*}
Where $\delta$ is Hochschild differential.
\begin{eqnarray*}
	As_{AAB}^{A} (a_1,a_2,b) &=& m_{BB}^A (\cancel{m_{AA}^B}(a_1,a_2),b) \underbrace{+ m_{AB}^A (m_{AA}^A(a_1,a_2),b)}_{\bigstar_2}\\
					&& - m_{AB}^A(a_1,\cancel{m_{AB}^B} (a_2,b)) \underbrace{-m_{AA}^A(a_1,m_{AB}^A(a_2,b))}_{\bigstar_1}\\
				&=& \psi_{b}(a_1\cdot a_2) - a_1 \cdot \psi_{b}(a_2).
\end{eqnarray*}
\begin{eqnarray*}
	As_{ABA}^{A}  (a_1,b,a_2) &=&  m_{BA}^A(\cancel{m_{AB}^B}(a_1,b),a_2)  \underbrace{+ m_{AA}^A(m_{AB}^A(a_1,b),a_2)}_{\dagger_2}\\
					&& - m_{AB}^A(a_1,\cancel{m_{BA}^B}(b,a_2))\underbrace{ - m_{AA}^A(a_1,m_{BA}^A(b,a_2))}_{\dagger_1}\\
				 &=& \psi_{b}(a_1)\cdot a_2 - a_1\cdot \varphi_{b}(a_2).
\end{eqnarray*}
\begin{eqnarray*}
	As_{BAA}^{A} (b,a_1,a_2) &=& m_{BA}^A(\cancel{m_{BA}^B}(b,a_1),a_2)  \underbrace{+ m_{AA}^A(m_{BA}^A(b,a_1),a_2)}_{\ast_1}\\
					&& - m_{BB}^A(b,\cancel{m_{AA}^B}(a_1,a_2)) \underbrace{- m_{BA}^A(b,m_{AA}^A(a_1,a_2))}_{\ast_2}\\
				&=& \varphi_{b}(a_1)\cdot a_2 - \varphi_{b}(a_1 \cdot a_3).
\end{eqnarray*}

The vanishing of these three last components together with remark $\ref{remder}$ gives the Leibniz's rule for $(\psi - \varphi)(b)$ which is hence a derivation of $\A$.

\begin{eqnarray*}
	As_{AAA}^A (e) &=& m_{BA}^A(\cancel{m_{AA}^B}(e_1,e_2),e_3) + m_{AA}^A(m_{AA}^A(e_1,e_2),e_3)\\
					&& - m_{AB}^A(e_1,\cancel{m_{AA}^B}(e_2,e_3)) - m_{AA}^A(e_1,m_{AA}^A(e_2,e_3))\\
				&=& (a_1 \cdot a_2)\cdot a_3 - a_1\cdot (a_2 \cdot a_3).\\
\end{eqnarray*}

Also we finally check that other components do not contribute : 

\begin{eqnarray*}
	As_{ABB}^B (e) &=& m_{BB}^B(\cancel{m_{AB}^B}(e_1,e_2),e_3) + \cancel{m_{AB}^B}(m_{AB}^A(e_1,e_2),e_3)\\
					&& - \cancel{m_{AB}^B}(e_1,m_{BB}^A(e_2,e_3)) - \cancel{m_{AA}^B}(e_1,m_{BB}^A(e_2,e_3))\\
				&=& 0.
\end{eqnarray*}
\begin{eqnarray*}
	As_{AAB}^B (e) &=& m_{BB}^B(\cancel{m_{AA}^B}(e_1,e_2),e_3) + \cancel{m_{AB}^B}(m_{AA}^A(e_1,e_2),e_3)\\
					&& - \cancel{m_{AB}^B}(e_1,\cancel{m_{AB}^B}(e_2,e_3)) - \cancel{m_{AA}^B}(e_1,m_{AB}^A(e_2,e_3))\\
				&=& 0.
\end{eqnarray*}
\begin{eqnarray*}
	As_{ABA}^B (e) &=& \cancel{m_{BA}^B}(\cancel{m_{AB}^B}(e_1,e_2),e_3) + \cancel{m_{AA}^B}(m_{AB}^A(e_1,e_2),e_3)\\
					&& - \cancel{m_{AB}^B}(e_1,\cancel{m_{BA}^B}(e_2,e_3)) - \cancel{m_{AA}^B}(e_1,m_{BA}^A(e_2,e_3))\\
				&=& 0.
\end{eqnarray*}
\begin{eqnarray*}
	As_{BAA}^B (e) &=& \cancel{m_{BA}^B}(\cancel{m_{BA}^B}(e_1,e_2),e_3) + \cancel{m_{AA}^B}(m_{BA}^A(e_1,e_2),e_3)\\
					&& - m_{BB}^B(e_1,\cancel{m_{AA}^B}(e_2,e_3)) - \cancel{m_{BA}^B}(e_1,m_{AA}^A(e_2,e_3))\\
				&=& 0. 
\end{eqnarray*}
\begin{eqnarray*}
	As_{AAA}^B (e) &=& \cancel{m_{BA}^B}(\cancel{m_{AA}^B}(e_1,e_2),e_3) + \cancel{m_{AA}^B}(m_{AA}^A(e_1,e_2),e_3)\\
					&& - \cancel{m_{AB}^B}(e_1,\cancel{m_{AA}^B}(e_2,e_3)) - \cancel{m_{AA}^B}(e_1,m_{AA}^A(e_2,e_3))\\
				&=& 0.
\end{eqnarray*}

\end{proof}

\begin{remarque}
	Now one can see that the definition $\ref{defcoho}$, which may seem ad-hoc, appears naturally in this context.
\end{remarque}

We now finally define the algebra $L$ of the theorem $\ref{thmprinc}$.
\begin{definition}
	Let $\A$ and $\B$ two associative algebras such that $\A$ is a $\A\oplus \B$-bimodule via the action of $\A$ on itself. We define 
		$$
			L :=\bigoplus_{(m,n) \in \mathbb{N}\times\mathbb{N^*}}L^{m,n}
		$$ 
		where $L^{m,n} := Lin(\A^{\otimes m}\oplus (\A ^{\otimes (m-1)}\otimes \B )\oplus (\A^{\otimes (m-2)}\otimes \B \otimes \A)\oplus... \oplus \B^{\otimes n}, \A )$.
\end{definition}

\begin{proposition}
	$L$ is a sub-differential graded Lie algebra of $(C^{\bullet + 1}(\A \oplus \B, \A \oplus \B), [\;,\;], \delta)$ (where $[\;,\;]$ is the Gerstenhaber bracket and $\delta$ is the Hochschild differential).
\end{proposition}

\begin{proof}
 	First, one shows that $L$ is closed under Gerstenhaber bracket. Let $f  \in L^{m,n}$ then one can decompose $f$ as the sum of $f_{i,j}$ where $f_{i,j} \in Lin(\A\otimes \B \otimes \A\otimes \A \otimes...,\A) = L_{i,j}$ (where $\A$ appears $i$-times and $\B$ $j$-times). \\
	Let $f$ in $L^{m_1,n_1}$ and $g$ in $L^{m_2,n_2}$, then one has
	$$
		[f_{i,j},g_{k,l}] \in L_{i + k - 1, j + l}.
	$$
	Hence
	$$
		[f,g] \in L^{m_1 + m_2 - 1,n_1 + n_2},
	$$
	i.e. $L$ is closed under Gerstenhaber bracket.\\
	Next, one have to show that $C_>$ is closed under $\delta$. Thanks to proposition $(\ref{propdoub})$, we have $\delta = (-1)^{*-1} [ m_{\A} + m_{\B},  ]$. Let $f  \in L^{m,n}$ 
	$$
		\delta f = (-1)^{max(m,n)-1}[ m_{\A} + m_{\B}, f ] = (-1)^{max(m,n)-1}([ m_{\A}, f] + [ m_{\B} , f ]).
	$$
	But if we decompose $f$ as before, the we have
	$$
		[ m_{\A}, f_{i,j}] \in L_{i + 1, j}
	$$
	and 
	$$
		[ m_{\B}, f_{i,j}] \in L_{i, j+1}.
	$$
	Hence $ \delta f_{i,j} \in L_{i+1,j+1}$, so we have $\delta f \in L^{m+1,n+1}$ and $L$ is closed under $\delta$.
\end{proof}

\begin{lemme} \label{lemmeMC}
	Let $\B$ and $\A$ two associative algebras on $B$, $A$, one has
		$$ 
			A_{BBB}^A + A_{BBA}^A +A_{BAB}^A +A_{ABB}^A +A_{AAB}^A +A_{ABA}^A +A_{BAA}^A = 0 
		$$ 
		$$ 
			\iff m_{BB}^A + m_{BA}^A + m_{AB}^A \in MC(L).
		$$
\end{lemme}

\begin{proof}
	Let $c = m_{BB}^A + m_{BA}^A + m_{AB}^A \in MC(L)$. One compute 
		$$
			(\delta c + \dfrac{1}{2} [ c, c ] \, )(e_1,e_2,e_3)= 0
		$$
	with $e_i = a_i + b_i \in \E = \A \oplus \B$.\\
	First, we have,
	$$
		\delta c = \,-( [ m_{\B} + m_{\A}, c ] )\, = \,-( [ m_{\B} , c ] \, + \, [  m_{\A}, c ]) \, 
	$$
	where $m_{\B}$ (resp. $m_{\A}$) is the multiplication on $\B$ (resp. on $\A$).
	Then we compute, 
	\begin{eqnarray*}
		[ m_{\B} , c ](e_1,e_2,e_3) &=& [ m_B , m_{BB}^A + m_{BA}^A + m_{AB}^A ](e_1,e_2,e_3) \\
						      &=& \underbrace{m_{BB}^A(m_{\B}(b_1,b_2),b_3) - m_{BB}^A(b_1,m_{\B}(b_2,b_3))}_{\bigtriangledown_1} \underbrace{+ m_{BA}^A ( m_{\B}(b_1,b_2),a_3)}_{\Box_1}\\
						      &&\underbrace{ - m_{AB}^A(a_1,m_{\B}(b_2,b_3))}_{\boxminus_1},
	\end{eqnarray*}
	\begin{eqnarray*}
		[ m_{\A}, c ](e_1,e_2,e_3) &=& [ m_A , m_{BB}^A + m_{BA}^A + m_{AB}^A ](e_1,e_2,e_3) \\
						      &=& \underbrace{m_{\A}(m_{BB}^A(b_1,b_2),a_3)}_{\Box_2} \underbrace{- m_{\A}(a_1,m_{BB}^A(b_2,b_3))}_{\boxminus _2}\\
						      && \underbrace{+ m_{\A}(m_{BA}^A(b_1,a_2),a_3) }_{\ast_1} \underbrace{- m_{\A}(a_1,m_{BA}^A(b_2,a_3))}_{\dagger_1} \underbrace{- m_{BA}^A(b_1,m_{\A}(a_2,a_3))}_{\ast_2}\\
						      && \underbrace{+ m_{\A}(m_{AB}^A(a_1,b_2),a_3)}_{\dagger_2} \underbrace{- m_{\A}(a_1,m_{AB}^A(a_2,b_3))}_{\bigstar_1}  \underbrace{+ m_{AB}^A(m_{\A}(a_1,a_2),b_3)}_{\bigstar_2}.
	\end{eqnarray*}
	Next, we have 
	\begin{eqnarray*}
		\dfrac{1}{2} [ c, c ](e_1,e_2,e_3) &=& [m_{BB}^A + m_{BA}^A + m_{AB}^A ,m_{BB}^A + m_{BA}^A + m_{AB}^A ] (e_1,e_2,e_3)\\
								 &=& \underbrace{- m_{BA}^A (b_1,m_{BB}^A (b_2,b_3)) + m_{AB}^A (m_{BB}^A (b_1,b_2),b_3)}_{\bigtriangledown_2} \underbrace{- m_{BA}^A (b_1,m_{BA}^A (b_2,a_3)) }_{\Box_3}\\
								 && \underbrace{- m_{BA}^A (b_1,m_{AB}^A (a_2,b_3))}_{\boxplus_1} \underbrace{+ m_{AB}^A (m_{BA}^A (b_1,a_2),b_3)}_{\boxminus_3} \underbrace{+ m_{AB}^A (m_{BA}^A (b_1,a_2),b_3)}_{\boxplus_2}.
	\end{eqnarray*}
	Then, we regroup the term marked by $\bigtriangledown, \Box, \boxminus, \ast, \dagger, \bigstar$ and $ \boxplus$ with the term marked by the same in proposition \ref{propas}, leading to the result.	
\end{proof}

\begin{corollaire}\label{corollairecocycle}
	$$
		Z^2_{nab}(\B,\A) \simeq MC(L)
	$$
\end{corollaire}

\begin{proof}
	By proposition \ref{prop1}, a 2-cocycle is the same as an extension $\E$ of $\B$ by $\A$. Then this extension is characterized by its multiplication, which is associative. So its associator vanishes. Since $\A$, $\B$ are associative algebras we have 
	$$
		A_{BBB}^A + A_{BBA}^A +A_{BAB}^A +A_{ABB}^A +A_{AAB}^A +A_{ABA}^A +A_{BAA}^A = 0
	$$
	By lemma \ref{lemmeMC} it is equivalent to $m_{BB}^A + m_{BA}^A + m_{AB}^A \in MC(L)$. 
\end{proof}

\subsection{Non-abelian cohomology as Deligne groupoid}
\begin{theorem}\label{thmprinc}
	$$
		H^2_{nab}(\B,\A) \simeq \mathcal{MC}(L)
	$$
\end{theorem}

\begin{proof}
	In corollary (\ref{corollairecocycle}) we have already seen that $Z^2_{nab}(\B,\A) \simeq MC(L)$. Then we must show that equivalence relation on 2-cocycles coincides with gauge relation on $MC(L)$(which is a dgL-algebra). One recall that two elements $l$ and $l'$ in $MC(L)$ are equivalent if there exists $\beta \in Lin(\B,\A)$ such that 
	$$
		l' = exp(ad_{\beta})l+g_{\beta}
	$$
	with
	$$
		g_{\beta} := - \sum_{n \in \mathbb{N}} \dfrac{1}{(n+1)!}(ad_{\beta})^n \delta \beta.
	$$
	Then we consider $l := \chi + \varphi + \psi$ and $e_i = a_i + b_i \in \A \oplus \B$. One computes,
	\begin{eqnarray*}
		exp(ad_{\beta})(\chi + \varphi + \psi)(e_1,e_2) &=& (\chi + \varphi + \psi + [\beta,\chi + \varphi + \psi] + \underbrace{\dfrac{1}{2}[\beta, [\beta ,\chi + \varphi + \psi]], ...}_{0})(e_1,e_2)\\
											     &=& \chi(e_1,e_2) + \varphi(e_1,e_2) + \psi(e_1,e_2) + [\beta,\chi + \varphi + \psi](e_1,e_2).
	\end{eqnarray*}
	But one has $[\beta,\chi]= 0$ since $\beta$ and $\chi$ both take values in $\A$. And,
	\begin{eqnarray*}
		[\beta, \varphi ](e_1,e_2) &=& -\cancel{\beta(\varphi(e_1,e_2))} - \cancel{\varphi(\beta(e_1),e_2)} - \varphi(e_1,\beta(e_2))\\
						      &=& - \varphi_{b_1}(\beta(b_2)).
	\end{eqnarray*}
	Similarly one has $[\beta, \psi](e_1,e_2) = - \psi_{b_2}(\beta(b_1))$.
	Hence we have :
	$$
		exp(ad_{\beta})(\chi + \varphi + \psi)(e_1,e_2) = \chi(b_1,b_2) + \varphi_{b_1}(a_2) + \psi_{b_2}(a_1) - \varphi_{b_1}(\beta(b_2))- \psi_{b_2}(\beta(b_1)). 
	$$
	Now, one computes $g_{\beta}$:
	$$
		-\delta \beta(e_1,e_2) = - [m_{\A} + m_{\B}, \beta ](e_1,e_2),
	$$
	with 
	$$
		 [m_{\A} ,\beta ](e_1,e_2) =  \beta(b_1)\cdot a_2+ a_1\cdot \beta(b_2) - \cancel{\beta(a_1 \cdot a_2)},
	$$
	and 
	$$
		[m_{\B} ,\beta ](e_1,e_2) =  \cancel{\beta(b_1)\cdot b_2} + \cancel{b_1\cdot \beta(b_2)}- \beta(b_1\cdot b_2).
	$$
	Therefore the following holds :
	$$
		-\delta \beta(e_1,e_2) = - \beta(b_1)\cdot a_2 - a_1 \cdot \beta(b_2) + \beta(b_1 \cdot b_2).
	$$
	Next we have 
	$$
		- [\beta,\delta \beta](e_1,e_2) = -[\beta ,  m_{\A}(\beta,\cdot) + m_{\A}(\cdot,\beta) - \beta(m_{\B})](e_1,e_2),
	$$
	but
	$$
		-[\beta ,  m_{\A}(\beta,\cdot) ] (e_1,e_2) =  - \beta(\cancel{(\beta(b_1) \cdot a_2)}) +  \cancel{\beta(\beta(b_1))}\cdot a_2+  m_{\A}(\beta(b_1),\beta(a_2)),
	$$
		$$
		-[\beta ,  m_{\A}(\cdot,\beta)  ] (e_1,e_2) =  - \beta(\cancel{a_1\cdot \beta(b_2)}) +  \beta(b_1)\cdot \beta(b_2)+ \beta(b_1) \cdot \cancel{\beta(\beta(a_2))},
	$$
	and $ [\beta , \beta(m_{\B})] = 0 $ since $\beta$ takes values in $\A$. We have therefore computed
	$$
		  -[\beta,\delta \beta](e_1,e_2) = 2 \beta(b_1)\cdot \beta(b_2).
	$$
	Now when $n\geq2$, since $\beta$ takes values in $\A$, $(ad_{\beta})^n = 0$, and we have :
	$$
		g_{\beta}(e_1,e_2) =  - \beta(b_1)\cdot a_2 - a_1\cdot \beta(b_2) + \beta(b_1\cdot b_2) + \beta(b_1)\cdot \beta(b_2). 
	$$
	Therefore, one can say $l'$ is equivalent to $l$ in $MC(L)$ if :
	\begin{eqnarray*}
		l'(e_1,e_2) &=& l(e_1,e_2) - \varphi_{b_1}(\beta(b_2))- \psi_{b_2}(\beta(b_1))  - \beta(b_1)\cdot a_2 - a_1\cdot \beta(b_2) \\
					&&+ \beta(b_1 \cdot b_2) + \beta(b_1)\cdot \beta(b_2).\\	
	\end{eqnarray*}
	In other words,
	\begin{eqnarray*}
		(\chi' + \varphi' + \psi') (e_1,e_2) &=& \varphi_{b_1}(a_2)  - \beta(b_1)\cdot a_2 + \psi_{b_2}(a_1) - a_1 \cdot \beta(b_2)\\
								  && + \chi(b_1,b_2) -  \varphi_{b_1}(\beta(b_2))- \psi_{b_2}(\beta(b_1))+ \beta(b_1 \cdot b_2) + \beta(b_1)\cdot \beta(b_2).
	\end{eqnarray*}
	On the one hand, when the two cocycles $(\chi , \varphi , \psi)$ and $(\chi' , \varphi' , \psi')$ are equivalent, equations (\ref{def1}), (\ref{def2}), and (\ref{def3}) are satisfied and so the previous equation is satisfied too. Hence $l$ and $l'$ are equivalent in $MC(L)$.\\
	On the other hand, when $l$ and $l'$ are equivalent, the previous equation is satisfied then equations (\ref{def1}), (\ref{def2}) and (\ref{def3}) also. Consequently the cocycles $(\chi , \varphi , \psi)$ and $(\chi' , \varphi' , \psi')$ are equivalent.
\end{proof}

\section{Link with abelian extensions}

\begin{definition}
	Let $\B$ be an associative algebra and $A$ a $\B$-bimodule. An abelian extension $\E$ of $\B$ by $A$ is as short exact sequence
	$$\begin{CD}
		0 @>>> A @>>> \E @>>> \B @>>> 0.
	\end{CD}$$
\end{definition}

Again, we can consider an equivalence relation on these extensions.

\begin{definition}
	Let $\E$ and $\E'$ two abelian extension of $\B$ by $A$. They are equivalent if there exists $\theta : \E \rightarrow \E' $ such that the following diagram commutes 
	$$\begin{CD}
		0 @>>> A @>>> \E @>>> \B @>>> 0\\
		@.          @|        @VV\theta V @|\\
		0 @>>> A @>>> \E' @>>> \B @>>> 0.
	\end{CD}$$
\end{definition}

\begin{proposition}
	There is a 1-1 correspondence between classes of abelian extension of $\B$ by $A$ and $C^2_{Hoch}(\B,A)$.
\end{proposition}

\begin{proof}
	This can be view as a special case of section 1.2. Indeed, we fixe $\varphi := m_{BA}^A $, $\psi := m_{AB}^A$, $\chi := m_{\B\B}^A$ and $m_{\A} = 0$.\\ In this case, equations $(\ref{defg})$, $(\ref{defd})$ and $(\ref{defbi})$ are satisfied since $A$ is a $\B$-bimodule,  equation $(\ref{defcocy})$ means that $\chi$ is a Hochschild 2-cocycle. Equations $(\ref{def1})$ and $(\ref{def2})$ become
	$$
		\varphi_{b}(a) = \varphi_b'(a) 
	$$
	and
	$$
		\psi_{b}(a) = \psi_{b}'(a).
	$$
	These equations mean that the bimodule structure does not change through equivalence relation. And finally, equation $(\ref{def3})$ just mean that $\chi$ and $\chi'$ differ by a Hochschild coboundary hence they are in the same cohomology class.
\end{proof}

\nocite{*}
\bibliographystyle{plain}
\bibliography{biblio}

\end{document}